\documentclass[12pt, reqno]{amsart}
\usepackage{amsfonts}
\usepackage{latexsym}
\usepackage{amssymb}
\usepackage{amsthm}
\usepackage{amsmath}
\usepackage{color}
\usepackage[all]{xy}

\numberwithin{equation}{section} \numberwithin{figure}{section}
\def\dmq{D^2\backslash Q}
\def\myev{\mathrm{ev}}
\def\fp{f_{\pi}}

 \newtheorem{thm}{Theorem}[section]
 
 \newtheorem{lem}[thm]{Lemma}
 \newtheorem{prop}[thm]{Proposition}
 \newtheorem{defn}[thm]{Definition}
 \theoremstyle{remark}
 
 \newtheorem{ex}[thm]{Example}

\title{A forcing relation of braids from Nielsen fixed point theory}

\author{Jiaoyun Wang}
\address{School of Science, Tianjin Chengjian University, Tianjin 300384, China}
\thanks{This work is partially supported by NSFC (11326077)}
\email{wanyueliang418@163.com}

\author{Xuezhi Zhao}
\address{School of Mathematical Sciences, Capital Normal University, Beijing 100048, China}
\email{zhaoxve@mail.cnu.edu.cn}

\keywords{braid group, fixed point, disk, fundamental group}

\begin{document}

\makeatletter

\let\uppercasenonmath\@gobble
\let\MakeUppercase\relax
\let\scshape\relax
\makeatother

\maketitle

\begin{abstract}  In this paper, we focus our attention on the connections between the
braid group and the Nielsen  fixed point theory. A new forcing relation between braids is introduced, and shown that it can be fulfilled by using Nielsen fixed point theory.
\end{abstract}

\section{Introduction}

The notion of braid type was proposed by Boyland \cite{Bo1,Boyland1994} and T. Matsuoka \cite{Ma} in the 1980's to study the dynamics of surface homeomorphisms. More specifically, forcing relations on the set of braid types were introduced to help us to understand the isotopy-stable dynamics, and it may be regarded as a two-dimensional generalization of Sharkovskii's relation for interval maps.
The interplay of braid types with the Artin braid groups has been used in a number of papers to try and understand the periodic orbit structure and Nielsen fixed point theory of surface homeomorphisms, see \cite{Gu} and more references there.

Since any forced braid is determined by a finite invariant set, which consists of periodic points of a surface homeomorphism $f$, it is natural to use fixed point theory to deal with the forcing relation. In this direction Jiang and Zheng \cite{JZ} deduced a trace formula for the computation of the $n+m$-strand forced extensions of a nontrivial braid  by using the Nielsen fixed point theory and a representation of braid groups, and shew that their forcing relation is computable algorithmically.

In this paper, we introduce a more delicate forcing relation by a careful study on the braids coming from fixed points. Moreover, the braids from fixed points of all iterations of a given homeomorphism on the disk are arranged into the same group $B^n_{n+1}$, so as a by-product we obtain a uniform coordinate set for all fixed points of all homeomorphisms with $n$-point invariant set.  In \cite{Gu}, Guaschi tried to use invariants of braids, hence of knots and links, to distinguish fixed point classes. What we are doing here in some sense is his inverse,we will use Nielsen fixed point theory to understand the forcing relation between braids. The connection between fixed points and braids are essentially used. It should be mentioned that the distinguishing of fixed point class for homeomorphisms on punctured disk is known to be solvable by a significant of Bridson and  Groves \cite{Bridson}, and Bogopolski and Maslakova \cite{BM2016}. Thus, our forcing relation between braids is also solvable, so is that in the sense of braid type of Boyland.

This paper is arranged as follows. In section 2, we fix some notations used in whole paper, and review some basic facts about braid groups. The relation between braid and invariants of an isotopy is specified. A new coordinate for fixed point class is introduced in section 3, which is derived for the relation between braids and fixed point of maps on the disk. Hence, a new forcing relation can be formalized. Section 4 will characterize the forced braid by stratified fixed point theory.  An Example is given in the final section.

\section{Preliminaries of braid groups}

We shall introduce some notations related to braid groups, some basic related facts will also be reviewed. The basic references for this topic are \cite{A, Bir, KT, Mu}.
We make following convention in notations.
\begin{itemize}
  \item $q_0$: the chosen base point $(0,1)$ of unit disk $D^2$,
  \item $q_k$: the chosen point $(-\frac{1}{k+1},0)$ in $D^2$ for $k=1,2, \ldots$,
  \item $\omega_k$: the line segment from $q_0$ to $q_k$  for $k=1,2, \ldots$,
  \item $\nu_k$: a loop in $D^2$ around $q_k$ with base point $q_0$  for $k=1,2, \ldots$,
  \item $Q$: the $n$-point set $\{q_1, \ldots, q_n\}$ of $D^2$,
  \item $f$: an orientation-preserving homeomorphism on $D^{2}$ such that $f(Q)=Q$,
  \item $\{h_t\}_{t\in I}$: an isotopy from the identity to $f$,
  \item $\omega_f$: a path in $\partial D^2$ from $q_0$ to $f(q_0)$ defined by $\omega_f(t)=h_t(q_0)$.
\end{itemize}

Based on convention above, we regard $B_n$ as the fundamental group $\pi_1(C_n(D^2), [q_1, q_2, \ldots, q_n])$ of the configuration space, i.e.
\begin{equation}\label{configandbraid}
B_n \cong \pi_1(C_n(D^2), [q_1, q_2, \ldots, q_n])
\end{equation}
where $C_n(D^2)$ is the orbit space $\{(u_{1}, u_{2}, \ldots, u_{n} )\in (D^{2})^n\mid u_{i}\neq u_{j}, \forall i\neq j\}/\Sigma_{n}$, where the symmetric group $\Sigma_{n}$ on $n$ letters admits a free action of ~$\Sigma_{n}$ on $(D^{2})^n$ given by $\mu\cdot(u_{1}, u_{2}, \ldots, u_{n} )=(u_{\mu(1)}, u_{\mu(2)}, \ldots, u_{\mu(n)} )$.

Note that there is a natural homomorphism ~$\mu: B_{n}\rightarrow \Sigma_{n}$ from the braid group ~$B_{n}$ to the symmetric group ~$\Sigma_{n}$ defined by ~$\mu(\sigma_{i})=(i, i+1), ~1\leq i\leq n-1$. Its kernel $P_{n}=ker\mu$ is called {\em the pure braid group}. It is known that $P_n$ has generators
 ~$A_{ij}=\sigma_{j-1}\sigma_{j-2}\ldots
\sigma_{i+1}\sigma_{i}^{2}\sigma_{i+1}^{-1}\ldots
\sigma_{j-2}^{-1}\sigma_{j-1}^{-1}, (1\leq i< j\leq n)$.
For each ~$2\leq j\leq n, ~U_{j}=\langle A_{ij}\mid 1\leq i <j \rangle$ is a free subgroup of rank ~$j-1$ of ~$B_{n}$.
We write $B_{n+s}^{n}$ the subgroup of $B_{n+s}$ consisting of those
elements for which induced automorphism of the symmetric group ~$\Sigma_{n+s}$ fixed ~$(n+1, \ldots, n+s)$ point-wisely.
In this paper, we focus on the subgroup $B_{n+1}^n$ of $B_{n+1}$.

\begin{lem}\label{biandetoushe}
(see \cite[Page 19, 23, 24]{Bir}) There is an exact sequence
\begin{equation}\label{eq}
1\to U_{n+1} \to B_{n+1}^n \stackrel{p_1}{\to} B_n\to 1,
\end{equation}
with return homomorphism $\iota_{1}: B_{n}\hookrightarrow B_{n+1}^{n}$, which is defined by
\begin{equation}\label{eq:return}
    [\gamma^{(1)}, \ldots, \gamma^{(n)}] \mapsto
    [c_{q_1},\ldots, c_{q_n}, \omega_{n+1}^{-1}]
    \ast [\gamma^{(1)}, \ldots, \gamma^{(n)}, c_{q_0}]
    \ast [c_{q_1},\ldots, c_{q_n}, \omega_{n+1}].
\end{equation}
Furthermore, $U_{n+1}$ is the free group of rank $n$, which is identified with $\pi_1(\dmq , q_0)$ by
\begin{equation}\label{eq:U}
     \gamma \stackrel{\phi}{\mapsto}
    [c_{q_1},\ldots, c_{q_n}, \omega_{n+1}^{-1}\gamma\omega_{n+1}],
\end{equation}
especially $\phi(\nu_i) = A_{i, n+1}$. Here $c_{q_{i}}$ is the constant path at $q_{i} (i=1,2,...n)$ and $\{\nu_{1},...,\nu_{n}\}$ is the set of generators of $\pi_1(\dmq , q_0)$.
\end{lem}

Hence, the braid group ~$B_{n+1}^{n}$ is a semi-direct product of $U_{n+1}$ and ~$B_{n}$. Our convention of $q_i$'s makes $A_{i,n+1}$ to be the usual geometric braid  if we look at the cylinder $D^2\times I$ along the vector $(0,1,0)$.

\begin{defn}
An isotopy $\{h_t\}_{t\in I}: D^2\to D^2$ between two orientation-preserving homomorphisms on $(D^2, Q)$ is said to determine an $n$-strand braid $\beta$ if
$$
\langle\myev_{Q}\circ h_t\rangle = \beta \in \pi_1(C_n(D^2), [q_1, \ldots, q_n]),$$
where $h_t$ is regarded as a path in space $\mathrm{Home}(D^2)$ of homeomorphisms on $D^2$, and $\myev_{Q}: \mathrm{Home}(D^2)\to C_n(D^2)$ is the evaluation map given by $g\mapsto [g(q_1),\ldots, g(q_n)]$.
\end{defn}

If an isotopy $\{h_t\}_{t\in I}$ determines an $n$-strand braid, then the set $Q$ must be invariant under homeomorphisms on both ends of such an isotopy, i.e. $h_0(Q)=h_1(Q)=Q$.

\begin{lem}\label{betam}
Let $\{h_t\}_{t\in I}: D^2\to D^2$ be an isotopy from the identity to a homeomorphism $f: (D^2, Q)\to (D^2, Q)$, determining an $n$-strand braid $\beta$. Then for any positive integer $m$, the isotopy $h_t\ast(f\circ h_t)\ast\cdots\ast(f^{m-1}\circ h_t)$ determines the braid $\beta^m$, where $\ast$ means the join of isotopies.
\end{lem}

\begin{proof}
We shall prove this lemma by induction on $m$. It is obviously true if $m=1$. Assume inductively that the isotopy $h_t\ast(f\circ h_t)\ast\cdots\ast(f^{k-1}\circ h_t)$ determines the braid $\beta^k$, i.e.
$$\myev_{Q}(h_t\ast(f\circ h_t)\ast\cdots\ast(f^{k-1}\circ h_t)) = \beta^k.$$
Note that the correspondence
$$(s,t)\mapsto(h_t\ast(f\circ h_t)\ast\cdots\ast(f^{k-1}\circ h_t))\circ h_s$$
gives us a map from $I\times I$ to $\mathrm{Home}(D^2)$. The four parts of the boundary of $I\times I$ yields a homotopy of paths with end-points keeping, i.e.
$$((h_t\ast(f\circ h_t)\ast\cdots\ast(f^{k-1}\circ h_t))\circ h_0)\ast (f^k\circ h_s) \dot{\simeq }
h_s \ast ((h_t\ast(f\circ h_t)\ast\cdots\ast(f^{k-1}\circ h_t))\circ h_1).$$
Since $h_1=f$ and $h_0=id$, we obtain that
$$h_t\ast(f\circ h_t)\ast\cdots\ast(f^{k-1}\circ h_t)\ast (f^k
 \circ h_s) \dot{\simeq }
h_s \ast ((h_t\ast(f\circ h_t)\ast\cdots\ast(f^{k-1}\circ h_t))\circ f).$$
It follows that
\begin{eqnarray*}
  & &  \myev_{Q}(h_t\ast(f\circ h_t)\ast\cdots\ast(f^{k-1}\circ h_t)\ast (f^k
 \circ h_s)) \\
 &\dot{\simeq }&
   \myev_{Q} (h_s \ast ((h_t\ast(f\circ h_t)\ast\cdots\ast(f^{k-1}\circ h_t))\circ f)\\
 & = &
   \myev_{Q} (h_s) \ast \myev_{Q}((h_t\ast(f\circ h_t)\ast\cdots\ast(f^{k-1}\circ h_t))\circ f).
\end{eqnarray*}
Since $f(Q) = Q$, we have that $\myev_{Q}((h_t\ast(f\circ h_t)\ast\cdots\ast(f^{k-1}\circ h_t))\circ f)= \myev_{Q}(h_t\ast(f\circ h_t)\ast\cdots\ast(f^{k-1}\circ h_t))$, which is $\beta^k$ by induction hypothesis. Note that $\myev_{Q} (h_s) =\beta$. Thus, we are done.
\end{proof}

\section{Coordinates for fixed point classes}

In this section, we shall give a kind of braid coordinates for fixed point and hence fixed point class.
Let $f$ be an orientation-preserving homeomorphism on $D^{2}$ such that $f(Q)=Q$.
We extend the idea of \cite{Gu} into the fixed points of all iterations of the given map $f$. Moreover, all braid coordinates lie in the same group $B^n_{n+1}$. More concrete constructions are given by using the identification in Lemma~\ref{biandetoushe}.

Let us recall some facts about fixed point coordinates in the fundamental group, see \cite{J3} for more details. Consider a fixed point $y$ of $f$ on $\dmq $, pick a path $\eta_y$ in $\dmq$ from the base point $q_0$ to $y$, then the $\fp$-conjugacy class of the element $\langle\omega_f(f\eta_y)\eta_y^{-1}\rangle\in \pi_1(\dmq, q_0)$ is independent of the choices of path $\eta_y$, and therefore is said to be the coordinate of $y$ in $\pi_1(\dmq , q_0)$, where $\omega_f$ is the chosen path from $q_0$ to $f(q_0)$. Two elements $\gamma'$ and $\gamma$ are said to be $\fp$-conjugate if $\gamma'=\fp(\alpha)\gamma\alpha^{-1}$ for some $\alpha\in \pi_1(\dmq , q_0)$. Here, $\fp : \pi_1(\dmq , q_0)\to \pi_1(\dmq , q_0)$ is defined by $\alpha\mapsto \omega_f \alpha \omega_f^{-1}$. Two fixed points $y$ and $y'$ are said to be in the same Nielsen fixed point class, denoted  $y\sim_N y'$, if they have the same coordinate in $\pi_1(\dmq , q_0)$. In all, there is an injective correspondence
\begin{eqnarray}
\varrho_{m}: \mathrm{Fix}(f^m|_{\dmq})/\sim_N \to \pi_1(\dmq , q_0)/\sim_{f^m_{\pi}}.
\end{eqnarray}

Let us consider the braids derived from fixed points of $f^m$.

\begin{lem}\label{dyz}
For any positive integer $m$, the correspondence
$$\mathrm{Fix}(f^m)\ni y\mapsto
 \langle [c_{q_1}, \ldots, c_{q_n}, \eta]
 \ast \myev_{Q\cup\{y\}}(h_t\ast(f\circ h_t)\ast\cdots\ast(f^{m-1}\circ h_t))
  \ast [c_{q_1}, \ldots,  c_{q_n}, \eta^{-1}]
 \rangle
$$ gives rise to a well-defined map
$$\Theta_m: \mathrm{Fix}(f ^m|_{\dmq}) \to  B^{n}_{n+1}/\sim_U,$$
where $c_{q_i}$ is the constant path at $q_i (i=1,2,...n)$, $\eta$ is a path in $\dmq$ from $q_{n+1}$ to $y$, and  $\sim_U$ is the conjugacy relation in $B^n_{n+1}$ restricting the conjugation to elements from $U_{n+1}$.
\end{lem}

\begin{proof}
Since $f^m(y)=y$, the path
$$
 \langle [c_{q_1}, \ldots, c_{q_n}, \eta]
 \ast \myev_{Q\cup\{y\}}(h_t\ast(f\circ h_t)\ast\cdots\ast(f^{m-1}\circ h_t))
  \ast [c_{q_1}, \ldots,  c_{q_n}, \eta^{-1}]
 \rangle
$$
is a loop in the configuration space $C_{n+1}(D^2)$ with based point $[q_1,\ldots, q_{n+1}]$, which is obviously an element in $B^n_{n+1}$.

If we have another path $\eta'$ in $\dmq$ from $q_{n+1}$ to $y$, then
$$\Theta_m(y, \eta')
= [c_{q_1}, \ldots, c_{q_n}, \eta'\eta^{-1}] \,
\Theta_m(y, \eta)\,
[c_{q_1}, \ldots, c_{q_n}, \eta'\eta^{-1}]^{-1}.$$
By Lemma~\ref{biandetoushe}, we know that $[c_{q_1}, \ldots, c_{q_n}, \eta'\eta^{-1}] = \phi(\omega_{n+1}^{-1}\eta'\eta^{-1}\omega_{n+1})$. This implies that $[c_{q_1}, \ldots, c_{q_n}, \eta'\eta^{-1}]$ lies in $U_{n+1}$.
Thus, different choices of $\eta$ yield the equivalent relation $\sim_U$.
\end{proof}

It should be mentioned that for any $\mu, \nu\in U_{n+1}$,
$\mu\, \iota_{1}(\beta^m)\nu \mu^{-1}=\iota_{1}(\beta^m)\, \iota_{1}(\beta^{-m})\mu\, \iota_{1}(\beta^m)\nu \mu^{-1}$ is still in the form $\iota_{1}(\beta^{m})\mu'$ because $U_{n+1}$ is a normal subgroup of $B^n_{n+1}$.

\begin{thm}\label{pitobraid}

If a fixed point $y$ of $f^m$ on $\dmq$ has a coordinate $\gamma$ in $\pi_1(\dmq,q_{0})$, then the corresponding braid $\Theta_m(y)$ in $B^n_{n+1}$ is $\iota_{1}(\beta^m)\phi(\gamma)$, where $\iota_{1}, \phi$ are homomorphisms given in Lemma~\ref{biandetoushe}.
\end{thm}

\begin{proof}
Consider the case $m=1$. By definition of fixed point coordinate, there is a path $\eta_y$ from $q_0$ to $y$ such that $\gamma = \omega_f(f\eta_y)\eta^{-1}_y$. Hence $\omega^{-1}_{n+1}\eta_y$ is a path from $q_{n+1}$ to $y$. Note that the restriction of $h_t$ on $\eta_y$ gives a homotopy from $\eta_y$ to $f\eta_y$. It follows that $h_t(y)\dot{\simeq} \eta_y^{-1}\omega_f(f\eta_y)$. As a path in configuration space $C_{n+1}(D^2)$, we have that
\begin{eqnarray*}
& &
 \Theta_m(y)=
[h_t(q_1), \ldots, h_t(q_n), \omega^{-1}_{n+1}\eta_y h_t(y)\eta^{-1}_y\omega_{n+1}]\\
&\dot{\simeq} &
 [c_{q_1}, \ldots, c_{q_n}, \omega^{-1}_{n+1}]
 \ast[h_t(q_1), \ldots, h_t(q_n), c_{q_{0}}]
 \ast[c_{q_1}, \ldots, c_{q_n}, \omega_{n+1}] \\
& &
 \ \ast[c_{q_1}, \ldots, c_{q_n}, \omega^{-1}_{n+1}\eta_y h_t(y)\eta^{-1}_y\omega_{n+1}]\\
& \dot{\simeq} &
\iota_{1}(\beta)[c_{q_1}, \ldots, c_{q_n}, \omega^{-1}_{n+1}\eta_y (\eta_y^{-1}\omega_f(f\eta_y))\eta^{-1}_y\omega_{n+1}]\\
& \dot{\simeq} &
\iota_{1}(\beta)[c_{q_1}, \ldots, c_{q_n}, \omega^{-1}_{n+1}\omega_f(f\eta_y)\eta^{-1}_y\omega_{n+1}]\\
& = &
\iota_{1}(\beta)[c_{q_1}, \ldots, c_{q_n}, \omega^{-1}_{n+1}\gamma\omega_{n+1}]\\
& = &
\iota_{1}(\beta)\phi(\gamma).
\end{eqnarray*}
The proof for case of $m>1$ is similar.
\end{proof}

We may call $\Theta_m(y)$  a fixed point coordinate of $y$, because we know from \cite[Theorem 1]{Gu} that two fixed points $y$ and $y'$ of $f^m$ are in the same fixed point class, i.e. $\varrho_m(y)=\varrho_m(y')$, if and only if $\Theta_m(y)=\Theta_m(y')\in B^n_{n+1}/\sim_U$.

Another useful coordinate for fixed point class comes from the mapping torus, see \cite{J3}. In our case, the mapping torus
$T_{f|_{\dmq}}$ of $f|_{\dmq}$ is defined to be the space $(\dmq \times \mathbb{R})/(x,
s+1)\sim (f(x), s)$, where ~$x\in \dmq, s\in R$. The equivalence class of ~$(x, s)$ will be written as ~$[x, s]$. By the Van~Kampen theorem, we have
\begin{equation}\label{eqmappingtorus}
\pi_{1}(T_{f|_{\dmq}}, [q_{0},0])
  =
 \pi_{1}(D^{2}\backslash Q, q_{0})* \mathbb{Z}\big/
   \langle \delta\tau=\tau \fp (\delta), \delta\in
     \pi_{1}(D^{2}\backslash Q, q_{0})\rangle
\end{equation}
where $\fp : \pi_{1}(D^{2}\backslash Q, q_{0})\rightarrow
\pi_{1}(D^{2}\backslash Q, q_{0})$ is the automorphism which is induced by $\gamma\mapsto \omega_f\gamma\omega_f^{-1}$, and $\tau$ be the time-$1$ loop ~$\{[q_{0}, t]\}_{0\leq t\leq 1}$ generating $\mathbb{Z}$. For each positive integer $m$, any fixed point $y$ of $f^m$ on $\dmq$ has a well-defined coordinate $\{[y, mt]\}_{0\le t \le 1}$ in the set $\pi_1(T_{f|_{\dmq}}, [q_0,0])/\mathrm{conj}$, which is actually $\tau \gamma$ if $y$ has coordinate $\gamma$ in $\pi_1(\dmq, q_0)$.

We summarize the relations among all mentioned coordinates  as follows.

\begin{prop}
There is a following commutative diagram
$$
\xymatrix{
  \mathrm{Fix}(f^m)/\sim_N \ar[dr]_{\Theta_m} \ar[r]^{\varrho_m}
    &\pi_1(\dmq , q_0)/\sim_{\fp^m}    \ar[r]^{i_*\ \ \ \ } \ar[d]^{\iota_{1}(\beta^m)\phi(\cdot)}
    & \pi_1(T_{f|_{\dmq}}, [q_0,0])/\mathrm{conj} \ar[dl]_{\psi}
    \\
                & B^n_{n+1}/\sim_U     &
}
$$
where $i_*$ is induced by the natural inclusion, $\psi$ is induced by an injective homomorphism given by $\tau\mapsto \beta$, $\nu_i\mapsto A_{i, n+1}$.
\end{prop}

The set  $\pi_1(T_{f|_{\dmq}}, [q_0,0])/\mathrm{conj}$ of conjugacy classes in $\pi_1(T_{f|_{\dmq}}, [q_0,0])$ contains all possible coordinates of fixed points of all iterations of the given homeomorphism $f$. The set $B^n_{n+1}/\sim_U$ contains coordinates of fixed points of all iterations of all homeomorphisms on $(D^2,Q)$. Thus, braids from different homeomorphisms can be compared. We can show our main definition.

\begin{defn}
A braid $\gamma\in B^n_{n+1}$ is said to be $(m, U)$-forced by $\beta\in B_n$ if, for any isotopy  $\{h_t\}_{t\in I}: D^2\to D^2$ from the identity to a homeomorphism $f: (D^2, Q)\to (D^2, Q)$, determining an $n$-strand braid $\beta$, there is a fixed point $y$ of $f^m|_{\dmq}$ such that $\Theta_m(y) = \gamma\in B_{n+1}^n/\sim_U$.
\end{defn}

Let us give some remarks about general forcing relations. A fixed point may correspond to different braids even an isotopy is given. Thus, forcing relations were considered as those among ``braid type'', see \cite{Bo1,Boyland1994} or \cite{Ma}. A braid type is usually regarded as a conjugacy class of a braid. A difference by a power of full-twists is allowed if one considers homeomorphisms on whole plane or on $D^2$ but freely at $\partial D^2$. The conjugation of a given braid comes from the choices of base points in the configuration space $C_n(D^2)$, which is solved here by fixing base point $[q_1, \ldots, q_n]$.  Thus, our forcing relation is more delicate, because the relation $\sim_U$ is smaller than the conjugacy relation in $B^n_{n+1}$. An explicit example can be found  in \cite[Sec. 3.4]{Gu}, showing that two conjugate braids may not equivalent under $\sim_U$.

\section{Forced braids}

Clearly, a homeomorphism $f: (D^2, Q)\to (D^2, Q)$ can be regarded as a stratified one with respect to the stratification $\{\dmq, Q\}$. We shall use the concepts in fixed point theory for stratified maps (see \cite{JZZ}) to identify the braids forced by the given braid $\beta$.

\begin{defn}(see \cite[Definition 2.2]{JZ})
A fixed point class of $f^{m}$ is called degenerate if its coordinate is related to some lower stratum. Otherwise, it is called non-degenerate.
\end{defn}

\begin{thm}\label{forcedbraid}
The set of braids which are $(m, U)$-forced by $\beta$ is
$$
\{ \iota_{1}(\beta^m)\phi(\gamma) \mid \gamma
\mbox{ is the coordinate of essential non-degenerated class of $f^m$ on }
\dmq
\}.
$$
\end{thm}

\begin{proof}
By the lower bound theorem \cite[Theorem 4.4]{JZZ} for stratified map, each essential non-degenerated class must contain a fixed point of $f^m$ on $\dmq$. Note that the isotopy class of $f|_{\dmq}$ is totally determined by $\beta$, and that coordinates of essential non-degenerated class are invariant under any (stratified)-homotopy. We know that each essential non-degenerated class of $f$ on $\dmq$ gives a braid $(m, U)$-forced by $\beta$.

Conversely, by the main result in \cite{JG}, one can isotope $f$ into a homeomorphism $f'$ such that $f'^m$ has no inessential or degenerated fixed point class. This means that $\iota_{1}(\beta^m)\phi(\gamma)$ is not $(m, U)$-forced braid if $\gamma$ is the coordinate of an inessential or degenerated fixed point class.
\end{proof}

Now we give a characterization of degenerated fixed point class. Corresponding braids from degenerated classes are said to be a ``peripheral braid'' in \cite{JZ}.

\begin{lem}
A fixed point class of $f^m$ is degenerated if and only if it has coordinate $\lambda$ such that $\fp^m(\nu_i) = \lambda\nu_i\lambda^{-1}$ for some $i$ with $1\le i\le n$.
\end{lem}

\begin{proof}
We shall give a proof for $m=1$ because the proof for general $m$ is the same.

Apply \cite[Prop. 3.2]{JZZ} into our stratified map $f$, a fixed point $y$ of $f$ lies in a degenerated class if and only if there is a path $\eta$ from $y$ to some $q_i$ such that $\eta \simeq f\eta: I, 0, [0,1), 1\to D^2, y, \dmq, Q$. Thus, the conclusion is obvious if $f$ has no fixed point on $Q$. In this situation, any fixed point class of $f$ on $\dmq$ is non-degenerated.

Now we consider the case that $f$ has some fixed points on $Q$. Since fixed point coordinates are homotopy invariants, by a small isotopy, we may assume that there is a small closed neighborhood $B(q_i, \varepsilon)$ of $q_i$ such that $B(q_i, \varepsilon)\subset Fix(f)$ for all $i$ with $f(q_i) = q_i$.

Let $y$ be a fixed point of $f$ lying in a degenerated class. By \cite[Prop. 3.2]{JZZ} again, there is a path $\eta$ from $y$ to some $q_i$ such that $\eta \simeq f\eta: I, 0, [0,1), 1\to D^2, y, \dmq, Q$. Thus, $\omega_i\eta^{-1}$ is a path from base point $q_0$ to $y$. Then the fixed point coordinate of $y$ is $\omega_f(f\omega_i)(f\eta^{-1})\eta\omega_i^{-1}\dot{\simeq}\omega_f(f\omega_i)\omega_i^{-1}$, because $\eta \dot{\simeq} f\eta$.

More precisely, above paths are not lie in $\dmq$ and meet $q_i$ somewhere. But, we can make $\eta$ so that its terminal part coincide with $\omega_i$, i.e. $\eta(t)= \omega_i(t)\in B(q_i,\varepsilon)$ for $t$ near the value $1$.
Let $q'_i$ be the unique point in $\omega_i\cap \partial B(q_i, \varepsilon)$, and $\omega'_i$ be the sub-path of $\omega_i$ with terminal point $q'_i$. Then the coordinate of $y$ is actually $\omega_f (f\omega'_i)\omega'^{-1}_i$.

Let $\delta_i$ be the small anti-clockwise loop lying in $\partial B(q_i, \varepsilon)$ with base point $q'_i$. Then $\omega'_i\delta_i\omega'^{-1}_i = \nu_i$ in $\pi_1(\dmq, q_0)$, where $\{\nu_1, \ldots, \nu_n\}$ is the chosen generator set of $\pi_1(\dmq, q_0)$. Since $\fp$ is an automorphism induced by a homeomorphism $f$ preserving the orientation and since $f(q_i)=q_i$, we have that $\fp(\nu_i) = \lambda\nu_i\lambda^{-1}$ for some $\lambda$ in $\pi_1(\dmq, q_0)$. Note that $\fp(\omega'_i\delta_i\omega'^{-1}_i) =\omega_f(f\omega'_i)(f\delta_i)(f\omega'^{-1}_i)\omega_f^{-1} = \omega_f(f\omega'_i)\delta_i(f\omega'^{-1}_i)\omega_f^{-1}$. We obtain that
$$
\lambda\omega'_i\delta_i\omega'^{-1}_i\lambda^{-1}
=\omega_f(f\omega'_i)\delta_i(f\omega'^{-1}_i)\omega_f^{-1}\in \pi_1(\dmq, q_0).$$
Since $\pi_1(\dmq, q_0)$ is a free group and $\omega'_i\delta_i\omega'^{-1}_i$ is a non-trivial element, we obtain that $\omega_f (f\omega'_i)\omega'^{-1}_i = \lambda$, i.e. $y$ has fixed point coordinate $\lambda$.

Conversely, suppose that a fixed point $y$ of $f$ has coordinate $\lambda$, where $\fp(\nu_i) = \lambda\nu_i\lambda^{-1}$ for some $i$ with $1\le i\le n$. Clearly, $q_i$ is a fixed point of $f$. We may also assume that there is a small closed neighborhood $B(q_i, \varepsilon)$ of $q_i$ such that $B(q_i, \varepsilon)\subset Fix(f)$. Above argument shows that the fixed point $q'_i$ in $B(q_i, \varepsilon)-{q_i}$ has fixed point coordinate $\lambda$. Thus, $y$ and $q'_i$ are in the same fixed point class of $f$ on $\dmq$. Since $q'_i$ is degenerated, the point $y$ lies in degenerated class.
\end{proof}

The information of essential fixed point class is contained in the Reidemeister trace $R(\fp)$, which can be computed by using Fox calculus $[1]-[Tr(J(\fp))]$, see \cite{Fa}. In general, the obtained form is not compact, i.e. there may be two items which are $\fp$-equivalent but are not the same element.  We encounter the problem when we decide which are degenerated. Such problem is called twist conjugation problem, which is a long-standing problem in geometric group theory and is finally solved in \cite{Bridson} and \cite{BM2016} independently.

In some situation, people are also interested in forcing relations when homeomorphisms on $D^2$ point-wisely fixed on the boundary $\partial D^2$. This is almost the same procedure as long as we rule out the braids $\iota_{1}(\beta^m)$, because we have

\begin{prop}
Any fixed point of $f^m$ on $\partial D^2$ has coordinate $1$ in $\pi_1(\dmq,q_0)$.
\end{prop}

\begin{proof}
Since we consider only orientation-preserving homeomorphisms on $(D^2, Q)$, the restriction $f^m|_{\partial D^2}$ of $f$ on  $\partial D^2$ is homotopic to the identity on $\partial D^2$. Any fixed point of $f^m$ on $D^2$ has the same coordinate as $q_0$, which clearly has coordinate $1$ in $\pi_1(\dmq, q_0)$.
\end{proof}

\section{An example}

We shall give an example to illustrate how to determine $(m, U)$-forcing relation.

\begin{lem}\label{tonggou}
(see \cite[Theorem 1.10]{Bir})
There is a canonical isomorphism $\xi: B_n \to \mathrm{MCG}(D^2, \mathrm{rel}\, Q, \partial D^2)$ such that for each generator $\sigma_i$ of $B_n$, the isomorphism on
$\pi_{1}(\dmq, q_{0})$ induced by $\xi(\sigma_i)$ is given by
\begin{equation}\label{Bn-action}
 \xi_{(\sigma_{i})_\pi}(x_j) = \left\{
  \begin{array}{ll}
     x_{i}\mapsto x_{i}x_{i+1}x_{i}^{-1} & \mbox{ if } j=i, \\
     x_{i+1}\mapsto x_{i} & \mbox{ if } j = i+1, \\
     x_{j}\mapsto x_{j}  & \mbox{ if } j\neq i,\ i+1,
  \end{array}
  \right.
\end{equation}
where $x_i$, $i=1, \ldots, n$, is determined by the loop $\nu_i$, and $\mathrm{MCG}(D^2, \mathrm{rel}\, Q, \partial D^2)$ is the mapping class group consisting of homeomorphisms on $D^{2}$ which fix $Q$ set-wisely and $\partial D^{2}$ point-wisely.
\end{lem}

\begin{ex}
$\beta=\sigma_{1}\sigma_{2}\sigma_{3}^{-1}\sigma_{4}^{-1}$.
\end{ex}

By Lemma~\ref{tonggou}, the automorphism $\fp$ is given by
$$   x_{1} \mapsto x_{1}x_{2}x_{5}x_{2}^{-1}x_{1}^{-1},\
     x_{2} \mapsto x_{1},\
     x_{3} \mapsto x_{2},\
     x_{4} \mapsto x_{5}^{-1}x_{3}x_{5},\
     x_{5} \mapsto x_{5}^{-1}x_{4}x_{5}.
$$
Then we obtain
$$
\begin{array}{lll}
  \frac{\partial \fp(x_{1})}{\partial x_{1}}=1-x_{1}x_{2}x_{5}x_{2}^{-1}x_{1}^{-1},
& \frac{\partial \fp(x_{2})}{\partial x_{2}}=0,
& \frac{\partial \fp(x_{3})}{\partial x_{3}}=0, \\
  \frac{\partial \fp(x_{4})}{\partial x_{4}}=0,
& \frac{\partial \fp(x_{5})}{\partial x_{5}}=-x_{5}^{-1}+x_{5}^{-1}x_{4}.
\end{array}
$$
The Reidemeister trace %
\begin{eqnarray*}
R(f_{\pi})
 & = & [1]-[Tr(J(\fp))]\\
 & = & [1]-\sum_{i=1}^5[(\frac{\partial \fp(x_{i})}{\partial x_{i}})]\\
 & = & [x_{1}x_{2}x_{5}x_{2}^{-1}x_{1}^{-1}]+[x_{5}^{-1}]-[x_{5}^{-1}x_{4}]\\
 & = & [x_{1}]+[x_{5}^{-1}]-[1].
\end{eqnarray*}

By using abelianization, we have that $R^{Ab}(f_{\pi})=[\bar x_1]+[\bar x_1^{-1}]-[\bar 1]$. It follows that three classes are distinct even in abelianization because abelian Nielsen equivalent relation is generated by $\bar x_i\sim \bar x_j$ for $i,j=1, \ldots, 5$. Here, $\bar x_i$ is the element in $H_1(\dmq)$ determined by $x_i$.  Since no $\bar x_i$ is fixed by $f_*: H_1(\dmq)\to H_1(\dmq)$, $f$ has no fixed point on $Q$. Thus, all fixed point class of $f$ are non-degenerated. By Theorem~\ref{forcedbraid}, each of these three classes gives a braid which is $(1,U)$-forced by $\beta=\sigma_{1}\sigma_{2}\sigma_{3}^{-1}\sigma_{4}^{-1}$.
By Theorem~\ref{pitobraid}, the braids corresponding to~$[x_{1}]$ is~$\iota_{1}(\beta) A_{16}$,
the braid corresponding to ~$[x_{5}^{-1}]$ is~$\iota_{1}(\beta) A_{56}^{-1}$,
and the braid corresponding to ~$[1]$ is ~$\iota_{1}(\beta)$.

\end{document}